\documentclass[12pt,reqno]{amsart}
\usepackage{eurosym}
\usepackage{amssymb}
\usepackage{amscd}
\usepackage{amsmath}
\usepackage{amsfonts}
\usepackage{color}
\usepackage{caption}
\usepackage{wrapfig}
\usepackage{hyperref}

\newcommand{\R}{\mathbb{R}}

\theoremstyle{plain}
\newtheorem*{acknowledgment}{Acknowledgment}

\newtheorem{algorithm}{Algorithm}[section]

\newtheorem{corollary}[algorithm]{Corollary}

\newtheorem{definition}[algorithm]{Definition}
\newtheorem{circlebundlelemma}[algorithm]{Circle Bundle Lemma}
\newtheorem{diskbundlelemma}[algorithm]{Disk Bundle Lemma}

\newtheorem{lemma}[algorithm]{Lemma}
\newtheorem*{maintheorem}{Main Theorem}
\newtheorem{theorem} [algorithm] {Theorem}

\newtheorem{proposition}[algorithm]{Proposition}
\newtheorem{remark}[algorithm]{Remark}

\numberwithin{equation}{algorithm}
\input{tcilatex}
\usepackage{tikz}
\usetikzlibrary{calc,fadings,decorations.pathreplacing,patterns,matrix,arrows}
\newcommand\pgfmathsinandcos[3]{%
  \pgfmathsetmacro#1{sin(#3)}%
  \pgfmathsetmacro#2{cos(#3)}%
}
\newcommand\LongitudePlane[3][current plane]{%
  \pgfmathsinandcos\sinEl\cosEl{#2} 
  \pgfmathsinandcos\sint\cost{#3} 
  \tikzset{#1/.estyle={cm={\cost,\sint*\sinEl,0,\cosEl,(0,0)}}}
}
\newcommand\LatitudePlane[3][current plane]{%
  \pgfmathsinandcos\sinEl\cosEl{#2} 
  \pgfmathsinandcos\sint\cost{#3} 
  \pgfmathsetmacro\yshift{\cosEl*\sint}
  \tikzset{#1/.estyle={cm={\cost,0,0,\cost*\sinEl,(0,\yshift)}}} %
}
\newcommand\LatitudePlaneT[3][current plane]{%
  \pgfmathsinandcos\sinEl\cosEl{#2} 
  \pgfmathsinandcos\sint\cost{#3} 
  \pgfmathsetmacro\yshift{\cosEl*\sint}
  \tikzset{#1/.estyle={cm={\cost,0,0,\cost*\sinEl,(0,\yshift+1.25)}}} %
}
\newcommand\LongitudePlaneT[3][current plane]{%
  \pgfmathsinandcos\sinEl\cosEl{#2} 
  \pgfmathsinandcos\sint\cost{#3} 
  \tikzset{#1/.estyle={cm={\cost,\sint*\sinEl,0,\cosEl,(0,1.25)}}}
}
\newcommand\DrawLongitudeCircleH[3][1]{
  \LongitudePlane{\angEl}{#2}
  \tikzset{current plane/.prefix style={scale=#1}}
  \pgfmathsetmacro\angVis{atan(sin(#2)*cos(\angEl)/sin(\angEl))} %
  \draw[current plane,very thick] (\angVis:1) arc(\angVis:180-#3-18:1) node[yshift=-0.45em,xshift=-0.06cm]{$H$};
\draw[current plane,very thick](180-#3-10:1) arc (180-#3-10:180-#3:1);
  \draw[current plane,dashed,very thick] (#3:1) arc (#3:\angVis:1);
  \draw[current plane](180-#3+8:1)node{$R$};
  \LatitudePlane{\angEl}{#3-8}
  \tikzset{current plane/.prefix style={scale=#1}}
  \pgfmathsetmacro\sinVis{sin(#2)/cos(#2)*sin(\angEl)/cos(\angEl)}
  \pgfmathsetmacro\angVis{asin(min(1,max(\sinVis,-1)))}
    \draw[<-,current plane] (210:1) arc (210:245:1);
    \draw[->,current plane] (255:1) arc (255:290:1);
}
\newcommand\DrawLongitudeCircleT[3][1]
{
\LongitudePlane{\angEl}{#2}
\tikzset{current plane/.prefix style={scale=#1}}
\pgfmathsetmacro\angVis{atan(sin(#2)*cos(\angEl)/sin(\angEl))}
\draw[current plane, thick] (#3:1) arc(#3: 90-#3/3:1) coordinate (t2); 
\draw[current plane, thick] (t2) arc (90-#3/3:90+#3/3:1)coordinate (t3);
\draw[current plane, thick] (t3) arc (90+#3/3:180-#3:1);  
\LatitudePlane{\angEl}{#3}
\tikzset{current plane/.prefix style={scale=#1}}
\pgfmathsetmacro\sinVis{sin(#2)/cos(#2)*sin(\angEl)/cos(\angEl)}
\pgfmathsetmacro\angVis{asin(min(1,max(\sinVis,-1)))}
\draw[current plane,thick] (162:1) arc (162:200:1) coordinate (t7);
\draw[current plane,thick] (t7) arc (200:245:1) coordinate (t1);
\draw[current plane,thick] (t1) arc(245:250:1);
\draw[current plane,dashed,thick] (95:1) arc (95:162:1);
\draw[current plane](120:1) coordinate (t6);
\draw[current plane,thick] (70:1) arc (70:75:1) coordinate (t4);
\draw[current plane,thick] (t4) arc (75:95.5:1);
\LongitudePlaneT{\angEl}{#2}
\tikzset{current plane/.prefix style={scale=#1}}
\pgfmathsetmacro\angVis{atan(sin(#2)*cos(\angEl)/sin(\angEl))} %
\draw[current plane, thick] (-#3:1) arc(-#3:-90+#3/3:1) coordinate (b2);
\draw[current plane,thick] (b2) arc(-90+#3/3:-90-#3/3:1) coordinate (b3);
\draw[current plane,thick] (b3) arc(-90-#3/3:#3-180:1);
\LatitudePlaneT{\angEl}{-#3}
\tikzset{current plane/.prefix style={scale=#1}}
\pgfmathsetmacro\sinVis{sin(#2)/cos(#2)*sin(\angEl)/cos(\angEl)}
\pgfmathsetmacro\angVis{asin(min(1,max(\sinVis,-1)))}
\draw[current plane,thick] (139:1) arc (139:200:1) coordinate (b7);
\draw[current plane,thick] (b7) arc (200:245:1) coordinate (b1);
\draw[current plane,thick] (b1) arc (245:250:1);
\draw[current plane,dashed,thick] (105:1) arc (105:139:1);
\draw[current plane,dashed,thick] (120:1) coordinate (b6);
\draw[current plane,thick] (70:1) arc (70:75:1) coordinate (b4);
\draw[current plane,thick] (b4) arc (75:105:1);
\draw[<->,densely dotted] (t1)--(b1);
\draw[<->,densely dotted] (t2)--(b2);
\draw[<->,densely dotted] (t3)--(b3);
\draw[<->,densely dotted] (t4)--(b4);
\draw[<->,loosely dotted] (t6)--(b6);
\draw[<->,densely dotted] (t7)--(b7);
}
\newcommand\DrawLongitudeCircleportion[2][1]{
  \LongitudePlane{\angEl}{#2}
  \tikzset{current plane/.prefix style={scale=#1}}
  \pgfmathsetmacro\angVis{atan(sin(#2)*cos(\angEl)/sin(\angEl))} %
\draw[current plane, pattern = dots] (3/4,.1)node[fill = white,above]{\small $D^2$}(0:1) arc (0:60:1) --(1/2,0) --(3/4,0)-- (1,0) ;
\draw[current plane, very thick](60:1) -- (1/2,0);
}

\newcommand\Drawcircleoverdisk[3][1]{
  \LongitudePlane{\angEl}{#2}
  \tikzset{current plane/.prefix style={scale=#1}}
  \pgfmathsetmacro\angVis{atan(sin(#2)*cos(\angEl)/sin(\angEl))} %
\draw[current plane,very thick](90:1)node[dot]{}node[below]{\small$p_n$} ({acos(#3)}:1)--(#3,{sin(acos(#3))/2})node[anchor = west]{\small$S^1$} --(#3,0) ;
}

\newcommand\Drawlightcircleoverdisk[3][1]{
  \LongitudePlane{\angEl}{#2}
  \tikzset{current plane/.prefix style={scale=#1}}
  \pgfmathsetmacro\angVis{atan(sin(#2)*cos(\angEl)/sin(\angEl))} %
\draw[current plane,dotted] (60:1) --(1/2,0) ;
}
\newcommand\DrawLatitudeCirclefront[2][1]{
  \LatitudePlane{\angEl}{#2}
  \tikzset{current plane/.prefix style={scale=#1}}
  \pgfmathsetmacro\sinVis{sin(#2)/cos(#2)*sin(\angEl)/cos(\angEl)}
  \pgfmathsetmacro\angVis{asin(min(1,max(\sinVis,-1)))}
  \draw[current plane] (170:1) arc (170:360:1);
  \draw[current plane,dashed] (0:1) arc (0:170:1);
}

\newcommand\DrawLatitudeCircle[2][1]{
  \LatitudePlane{\angEl}{#2}
  \tikzset{current plane/.prefix style={scale=#1}}
  \pgfmathsetmacro\sinVis{sin(#2)/cos(#2)*sin(\angEl)/cos(\angEl)}
  \pgfmathsetmacro\angVis{asin(min(1,max(\sinVis,-1)))}
  \draw[current plane] (\angVis:1) arc (\angVis:-\angVis-180:1);
  \draw[current plane,dashed] (180-\angVis:1) arc (180-\angVis:\angVis:1);
}
\newcommand\Drawlabels[2][1]{
  \LatitudePlane{\angEl}{#2}
  \tikzset{current plane/.prefix style={scale=#1}}
  \pgfmathsetmacro\sinVis{sin(#2)/cos(#2)*sin(\angEl)/cos(\angEl)}
  \pgfmathsetmacro\angVis{asin(min(1,max(\sinVis,-1)))}
  \draw[current plane](0,0)node[dot]{}(0,0)node[right,fill = white]{\small$p_0$}(180:1)node[dot]{}node[left]{$p_1$} (270:1)node[dot]{} node[below]{\small$p_2$} (360:1);
  \draw[current plane,dashed] (0:1)node[dot]{}node[right]{\small$A(p_1)$} (90:1)node[dot]{} node[above]{\small$A(p_2)$} (180:1);
}
\newcommand\DrawLatitudeCirclefilledD[2][1]{
  \LatitudePlane{\angEl}{#2}
  \tikzset{current plane/.prefix style={scale=#1}}
  \pgfmathsetmacro\sinVis{sin(#2)/cos(#2)*sin(\angEl)/cos(\angEl)}
  \pgfmathsetmacro\angVis{asin(min(1,max(\sinVis,-1)))}
\draw[current plane, very thick, pattern = dots] (0:1) arc (0:360:1);
  \draw[current plane] (-.4,0)node[fill = white]{\small$D^{n-1}$}(0:1) arc (0:180:1);
}
\tikzset{%
  >=latex, 
  inner sep=0pt,%
  outer sep=2pt,%
  mark coordinate/.style={inner sep=0pt,outer sep=0pt,minimum size=3pt,
    fill=black,circle}%
}
\begin{document}

\title{The Diffeomorphism Type of Manifolds with Almost Maximal Volume}
\author{Curtis Pro}
\address{{Department of Mathematics, University of California, Riverside} }
\email{cpro@math.toronto.edu}
\author{Michael Sill}
\address{{Department of Mathematics, University of California, Riverside} }
\email{msill@calbaptist.edu}
\author{Frederick Wilhelm}
\address{{Department of Mathematics, University of California, Riverside}}
\email{fred@math.ucr.edu}

\begin{abstract}
The smallest $r$ so that a metric $r$--ball covers a metric space $M$ is
called the radius of $M.$ The volume of a metric $r$-ball in the space form
of constant curvature $k$ is an upper bound for the volume of any Riemannian
manifold with sectional curvature $\geq k$ and radius $\leq r$. We show that
when such a manifold has volume almost equal to this upper bound, it is
diffeomorphic to a sphere or a real projective space.\bigskip
\end{abstract}

\maketitle

\section{Introduction}

Any closed Riemannian $n$--manifold $M$ has a lower bound, $k\in \mathbb{R},$
for its sectional curvature$.$ This gives an upper bound for the volume of
any metric ball $B\left( x,r\right) \subset M,$ 
\begin{equation*}
\mathrm{vol\,}B\left( x,r\right) \leq \mathrm{vol}\text{ }\mathcal{D}%
_{k}^{n}\left( r\right) ,
\end{equation*}%
where $\mathcal{D}_{k}^{n}\left( r\right) $ is an $r$--ball in the $n$%
--dimensional, simply connected space form of constant curvature $k.$ If $%
\mathrm{rad}\text{ }M$ is the smallest number $r$ such that a metric $r$%
--ball covers $M,$ it follows that 
\begin{equation}
\mathrm{vol\,}M\leq \mathrm{vol}\text{ }\mathcal{D}_{k}^{n}\left( \mathrm{rad%
}\text{ }M\right) .  \label{max vol inequal}
\end{equation}

The invariant $\mathrm{rad}\text{ }M$ is known as the \emph{radius }of $M$
and can alternatively be defined as 
\begin{equation*}
\mathrm{rad}M=\min_{p\in M}\max_{x\in M}\mathrm{dist}\left( p,x\right) .
\end{equation*}

In the event that $\mathrm{vol\,}M$ is almost equal to $\mathrm{vol}\text{ }%
\mathcal{D}_{k}^{n}\left( \mathrm{rad}\text{ }M\right) ,$ we determine the
diffeomorphism type of $M$.

\begin{maintheorem}
\label{Diff Stab}Given $n\in \mathbb{N},k\in \mathbb{R},$ and $r>0,$ there
is an $\varepsilon >0$ so that every closed Riemannian $n$--manifold $M$
with 
\begin{eqnarray}
\mathrm{sec}\text{ }M &\geq &k,  \notag  \label{maintheoremconditions} \\
\mathrm{rad}\text{ }M &\leq &r,\text{ and } \\
\mathrm{vol}\text{ }M &\geq &\mathrm{vol}\text{ }\mathcal{D}_{k}^{n}\left(
r\right) -\varepsilon  \notag
\end{eqnarray}%
is diffeomorphic to $S^{n}$ or $\mathbb{R}P^{n}.$
\end{maintheorem}

This generalizes Part 1 of Theorem A in \cite{GrovPet3}, where Grove and
Petersen classified these manifolds up to homeomorphism. They also showed
that for any $\varepsilon >0$ and $M=S^{n}$ or $\mathbb{R}P^{n},$ there are
Riemannian metrics that satisfy (\ref{maintheoremconditions}), except when $%
k>0$ and $r\in \left( \frac{1}{2}\frac{\pi }{\sqrt{k}},\frac{\pi }{\sqrt{k}}%
\right) .$ Thus Inequality (\ref{max vol inequal}) is optimal, except when $%
k>0$ and $r\in \left( \frac{1}{2}\frac{\pi }{\sqrt{k}},\frac{\pi }{\sqrt{k}}%
\right) .$

For $k>0$ and $r\in \left( \frac{1}{2}\frac{\pi }{\sqrt{k}},\frac{\pi }{%
\sqrt{k}}\right) ,$ Grove and Petersen also computed the optimal upper
volume bound for the class of manifolds $M$ with 
\begin{equation}
\mathrm{sec}\text{ }M\geq k~~~\text{ and }~~~\mathrm{rad}\text{ }M\leq r.
\label{lowercurvupperrad}
\end{equation}%
It is strictly less than $\mathrm{vol}\text{ }\mathcal{D}_{k}^{n}\left(
r\right) ,$ \cite{GrovPet3}. For $k>0$ and $r\in \left( \frac{1}{2}\frac{\pi 
}{\sqrt{k}},\frac{\pi }{\sqrt{k}}\right) ,$ manifolds satisfying (\ref%
{lowercurvupperrad}) with almost maximal volume are already known to be
diffeomorphic to spheres \cite{GrovWilh1}. The main theorem in \cite{OSY}
gives the same result when $r=\frac{\pi }{\sqrt{k}}.$

For $k>0$ and $r={\frac{\pi }{\sqrt{k}}},$ the maximal volume $\mathrm{vol}%
\text{ }\mathcal{D}_{1}^{n}\left( {\frac{\pi }{\sqrt{k}}}\right) $ is
realized by the $n$-sphere with constant curvature $k$. For $k>0$ and $r=%
\frac{\pi }{2\sqrt{k}},$ the maximal volume $\mathrm{vol\,}\mathcal{D}%
_{1}^{n}\left( \frac{\pi }{2\sqrt{k}}\right) $ is realized by $\mathbb{R}%
P^{n}$ with constant curvature $k.$ Apart from these cases, there are no
Riemannian manifolds $M$ satisfying (\ref{lowercurvupperrad}) and $\mathrm{%
vol}\text{ }M=\mathrm{vol}\text{ }\mathcal{D}_{k}^{n}\left( r\right) .$
Rather, the maximal volume is realized by one of the following two types of
Alexandrov spaces \cite{GrovPet3}.

\begin{definition}
\textbf{(Purse) }Let $R:\mathcal{D}_{k}^{n}\left( r\right) \rightarrow 
\mathcal{D}_{k}^{n}\left( r\right) $ be reflection in a totally geodesic
hyperplane $H$ through the center of $\mathcal{D}_{k}^{n}\left( r\right) $.
The Purse, $P_{k,r}^{n},$ is the quotient space%
\begin{equation*}
\mathcal{D}_{k}^{n}\left( r\right) /\left\{ v\sim R\left( v\right) \right\} ,%
\text{ provided }v\in \partial \mathcal{D}_{k}^{n}\left( r\right) .
\end{equation*}

Alternatively we let $\left\{ \frac{1}{2}\mathcal{D}_{k}^{n}\left( r\right)
\right\} ^{+}\cup \left\{ \frac{1}{2}\mathcal{D}_{k}^{n}\left( r\right)
\right\} ^{-}=D_{k}^{n}(r)$ be the decomposition of $\mathcal{D}%
_{k}^{n}\left( r\right) $ into the two half disks on either side of $H.$
Then $P_{k,r}^{n}$ is isometric to the double of $\left\{ \frac{1}{2}%
\mathcal{D}_{k}^{n}\left( r\right) \right\} ^{+}.$ In particular, $%
P_{k,r}^{n}$ is homeomorphic to $S^{n}.$
\end{definition}

\begin{definition}
\label{crosscapexample} \textbf{(Crosscap) }The constant curvature $k$
Crosscap, $C_{k,r}^{n},$ is the quotient of $\mathcal{D}_{k}^{n}\left(
r\right) $ obtained by identifying antipodal points on the boundary. Thus $%
C_{k,r}^{n}$ is homeomorphic to $\mathbb{R}P^{n}$. There is a canonical
metric on $C_{k,r}^{n}$ that makes this quotient map a submetry. The
universal cover of $C_{k,r}^{n}$ is the double of $\mathcal{D}_{k}^{n}\left(
r\right) $. If we write this double as $\boldsymbol{DD}_{k}^{n}\left(
r\right) \equiv \mathcal{D}_{k}^{n}\left( r\right) ^{+}\cup _{\partial 
\mathcal{D}_{k}^{n}(r)^{\pm }}\mathcal{D}_{k}^{n}\left( r\right) ^{-},$ then
the free involution%
\begin{equation*}
A:\boldsymbol{DD}_{k}^{n}\left( r\right) \left( r\right) \longrightarrow 
\boldsymbol{DD}_{k}^{n}\left( r\right) \left( r\right)
\end{equation*}%
that gives the covering map $\boldsymbol{DD}_{k}^{n}\left( r\right)
\longrightarrow C_{k,r}^{n}$ is%
\begin{equation*}
A:\left( x,+\right) \longmapsto \left( -x,-\right) ,
\end{equation*}%
where the sign in the second entry indicates whether the point is in $%
\mathcal{D}_{k}^{n}(r)^{+}$ or $\mathcal{D}_{k}^{n}(r)^{-}.$
\end{definition}

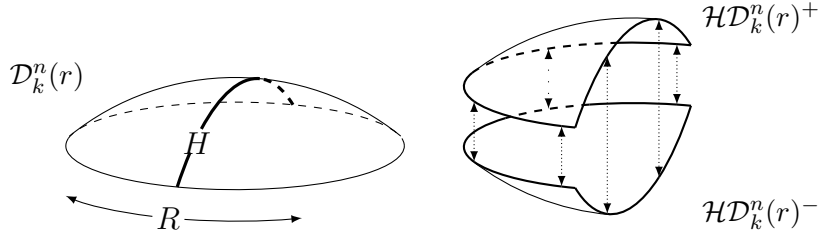
\begin{figure}[ht]
\vspace{-20pt}
\centering
\begin{tabular}{lr}
\begin{tikzpicture}
\def\R{3.5} 
\def\r{50}
\def\angEl{15} 
\draw (-2.5,3.5) node {\small$\mathcal{D}^n_{k}(r)$};
\draw (\r+1:\R) arc (\r+1:180-\r-1:\R);
\DrawLatitudeCirclefront[\R]{\r} 
\DrawLongitudeCircleH[\R]{70}{\r}
\end{tikzpicture} $~~~~~$ & \begin{tikzpicture}
\def\R{3.5} 
\def\r{50}
\def\angEl{15} 
\draw (1.7,3.5) node {\small$\mathcal{HD}^n_{k}(r)^+$};
\draw (1.7,.9) node {\small$\mathcal{HD}^n_{k}(r)^-$};
\draw(85:\R) arc (85:180-\r-2:\R);
\draw (0,4.37)+(265:\R) arc (265:180+\r+2:\R);
\DrawLongitudeCircleT[\R]{70}{\r}
\end{tikzpicture}
\end{tabular}
\caption{Two equivalent constructions of $P^{2}_{1,r}$}
\end{figure}

Let $\left\{ M_{i}\right\} _{i=1}^{\infty }$ be a sequence of closed $n$%
-manifolds with $\mathrm{sec}$ $M\geq k$, $\mathrm{rad\,}M\leq r,$ and $\{%
\mathrm{vol}\text{{}}M_{i}\}$ converging to $\mathrm{vol}\text{ }\mathcal{D}%
_{k}^{n}\left( r\right) ,$ where $r\leq {\frac{\pi }{2\sqrt{k}}}$ if $k>0$.
Grove and Petersen showed that $\{M_{i}\}$ has a subsequence that converges
to either the crosscap, $C_{k,r}^{n},$ or the purse, $P_{k,r}^{n},$ in the
Gromov-Hausdorff topology \cite{GrovPet3}. Our main theorem follows by
combining this with the following \textit{diffeomorphism stability theorems}.

\begin{theorem}
\label{Purse Stab}Let $\left\{ M_{\alpha }\right\} _{\alpha =1}^{\infty }$
be a sequence of closed Riemannian $n$--manifolds with $\mathrm{sec}$ $%
M_{\alpha }\geq k$ so that 
\begin{equation*}
M_{\alpha }\longrightarrow P_{k,r}^{n}
\end{equation*}%
in the Gromov-Hausdorff topology. Then all but finitely many of the $%
M_{\alpha }$s are diffeomorphic to $S^{n}.$
\end{theorem}

\begin{theorem}
\label{Cross Cap Stability}Let $\left\{ M_{\alpha }\right\} _{\alpha
=1}^{\infty }$ be a sequence of closed Riemannian $n$--manifolds with $%
\mathrm{sec}$ $M_{\alpha }\geq k$ so that 
\begin{equation*}
M_{\alpha }\longrightarrow C_{k,r}^{n}
\end{equation*}%
in the Gromov-Hausdorff topology. Then all but finitely many of the $%
M_{\alpha }$s are diffeomorphic to $\mathbb{R}P^{n}.$
\end{theorem}

Theorem $\ref{Cross Cap Stability}$ follows directly from Theorem 6.1 in 
\cite{KMS}, as all points in $C_{k,r}^{n}$ are $(n,0)$--strained. So this
paper is devoted to the proof of Theorem \ref{Purse Stab}. Our proof of
Theorem \ref{Purse Stab} is related to an alternative proof of Theorem $\ref%
{Cross Cap Stability}$ which was included in the original version of this
paper and is also in \cite{ProSillWilh}.

\begin{remark}
One can get Theorem \ref{Purse Stab} for the case $k=1$ and $r>\mathrm{arccot%
}\,\left( \frac{1}{\sqrt{n-3}}\right) $ as a corollary of Theorem C in \cite%
{GrovWilh2}. Theorem \ref{Cross Cap Stability} when $k=1$ and $r=\frac{\pi }{%
2}$ follows from the main theorem in \cite{Yam1} and the fact that $C_{1,%
\frac{\pi }{2}}^{n}$ is $\mathbb{R}P^{n}$ with constant curvature $1.$
\end{remark}

In \cite{GrovPet3}, Grove and Petersen proved the topological stability
theorems that are analogous to Theorems \ref{Purse Stab} and $\ref{Cross Cap
Stability}$. Perelman has since proved a much more general Topological%
\textbf{\ }Stability Theorem, which in particular implies the following.

\bigskip

\noindent \textbf{Topological Stability Theorem:} \emph{Let }$\{M_{\alpha
}\}_{\alpha }$\emph{\ be a sequence of closed Riemannian }$n$\emph{%
--manifolds with sectional curvature }$\geq k.$\emph{\ If the
Gromov-Hausdorff limit of }$\{M_{\alpha }\}_{\alpha }$\emph{\ is }$X$\emph{\
and }$\dim \left( X\right) =n,$\emph{\ then all but finitely many of the }$%
M_{\alpha }$\emph{'s are homeomorphic to }$X,$ \cite{Perel, Kap}$.$

\bigskip

In a similar way, Theorems \ref{Purse Stab} and $\ref{Cross Cap Stability}$
would follow from an affirmative answer to the following open question.

\bigskip

\noindent \textbf{Diffeomorphism Stability Question:} \emph{Let }$%
\{M_{\alpha }\}_{\alpha }$\emph{\ be a sequence of closed Riemannian }$n$%
\emph{--manifolds with sectional curvature }$\geq k.$\emph{\ If the
Gromov-Hausdorff limit of }$\{M_{\alpha }\}_{\alpha }$\emph{\ is }$X$\emph{\
and }$\dim \left( X\right) =n,$\emph{\ then are all but finitely many of the 
}$M_{\alpha }$\emph{'s diffeomorphic to each other }\cite{GrovWilh2}\emph{?}

\bigskip

An affirmative answer to the Diffeomorphism Stability Question would also
provide generalizations of Cheeger's Finiteness Theorem and the Diameter
Sphere Theorem \cite{Cheeg1}, \cite{Cheeg2}, \cite{GrovShio}, \cite%
{GrovWilh2}.

\begin{definition}
Let $\mathcal{M}_{k}\left( n\right) $ be the class of closed Riemannian $n$%
--manifolds with sectional curvature $\geq k.$ A compact, $n$--dimensional $%
X\in \mathrm{closure}\left( \mathcal{M}_{k}\left( n\right) \right) $ is
called \emph{diffeomorphically stable} if for any sequence $\left\{
M_{\alpha }\right\} _{\alpha =1}^{\infty }\subset \mathcal{M}_{k}\left(
n\right) $ with $M_{\alpha }\longrightarrow X,$ in the Gromov--Hausdorff
topology, all but finitely many of the $M_{\alpha }$s are diffeomorphic to
each other.
\end{definition}

Together, Theorem \ref{Purse Stab} and Corollary E of \cite{GrovWilh1} say
that purses and the so-called \textquotedblleft lemons\textquotedblright\ of 
\cite{GrovPet3} are diffeomorphically stable. These are the only known
diffeomorphically stable limit spaces having a space of directions that is
Gromov--Hausdorff far from the unit sphere.

The proof of Theorem \ref{Purse Stab} starts with the simple observation
that the purse, $P_{k,r}^{n},$ can be topologically identified with the
disjoint union of $D^{n-1}\times S^{1}$ and $S^{n-2}\times D^{2}$ glued
together via the identity map of their common boundary $S^{n-2}\times S^{1}.$
(See figure { \ref{pursefig})} Using this, we show that if $\left\{ M_{\alpha
}\right\} _{\alpha }$ is as in Theorem \ref{Purse Stab}, then for $\alpha $
sufficiently large, $M_{\alpha }$ is diffeomorphic to the disjoint union of $%
D^{n-1}\times S^{1}$ and $S^{n-2}\times D^{2}$ glued together via a
diffeomorphism $f$ of $S^{n-2}\times S^{1}.$ That is, $M_{\alpha }$ is
diffeomorphic to 
\begin{equation}
D^{n-1}\times S^{1}\cup _{f}S^{n-2}\times D^{2}.  \label{handle decop dfn}
\end{equation}%
We show, moreover, that the diffeomorphism $f:S^{n-2}\times
S^{1}\longrightarrow S^{n-2}\times S^{1}$ satisfies 
\begin{equation*}
p_{n-2}\circ f=p_{n-2},
\end{equation*}%
where 
\begin{equation*}
p_{n-2}:S^{n-2}\times S^{1}\longrightarrow S^{n-2}
\end{equation*}%
is projection to the first factor.

Notice that a diffeomorphism $f:S^{n-2}\times S^{1}\longrightarrow
S^{n-2}\times S^{1}$ so that $p_{n-2}\circ f=p_{n-2}$ gives rise to an
element of $\pi _{n-2}\left( \mathrm{Diff}_{+}\left( S^{1}\right) \right) .$
If two such diffeomorphisms give the same homotopy class, then the
construction (\ref{handle decop dfn}) yields diffeomorphic manifolds (cf. 
\cite{GrovWilh2}). Since the group of orientation preserving diffeomorphisms
of the circle deformation retracts to $SO\left( 2\right) ,$ it follows that $%
M_{\alpha }$ is diffeomorphic to $S^{n}$ for all $\alpha $ sufficiently
large.

To construct the decomposition (\ref{handle decop dfn}) we start with the
observation that the singularities of $P_{k,r}^{n}$ occur along a constant
curvature sphere of codimension $2,$ that we call $\mathcal{S}^{n-2}.$ The
construction of $P_{k,r}^{n}$ also allows us to view $\mathcal{S}^{n-2}$ as
the boundary of $\mathcal{D}_{k}^{n-1}\left( r\right) .$ As in \cite{OSY},
we then write coordinate functions $f_{i}$ of $\mathcal{D}_{k}^{n-1}\left(
r\right) $ in terms of distance functions from points of $\mathcal{S}^{n-2}.$
The formulas for these coordinate functions also make sense on $P_{k,r}^{n}$
and, as in \cite{OSY}, restrict to an isometric embedding of $\mathcal{S}%
^{n-2}$ into $\mathbb{R}^{n-1}.$ Since the $f_{i}$'s are written in terms of
distance functions, they have lifts, $f_{i}^{\alpha },$ to the $M_{\alpha }$%
's. Using these lifts, we define 
\begin{eqnarray*}
\Psi ^{\alpha } &:&M_{\alpha }\longrightarrow \mathbb{R}^{n-1} \\
\Psi ^{\alpha } &=&\left( f_{1}^{\alpha },f_{2}^{\alpha },\ldots
,f_{n-1}^{\alpha }\right) .
\end{eqnarray*}%
We then show that the restriction of $\Psi ^{\alpha }$ to a subset $%
E_{D}^{\alpha }\subset M_{\alpha }$ is a trivial $S^{1}$--bundle over $%
D^{n-1},$ and the restriction of $\Psi ^{\alpha }$ to $M\setminus \mathrm{int%
}\left( E_{D}^{\alpha }\right) $ is a trivial $D^{2}$--bundle over $S^{n-2}.$
In other words, $E_{D}^{\alpha }\cong D^{n-1}\times S^{1}$ and $M\setminus 
\mathrm{int}\left( E_{D}^{\alpha }\right) \cong S^{n-2}\times D^{2},$ as in (%
\ref{handle decop dfn}).

\begin{remark}
It was shown in \cite{GrovWilh2} that the presence of a decomposition of the
form $M_{\alpha }=D^{n-1}\times S^{1}\cup _{f}S^{n-2}\times D^{2}$ with $%
p_{n-2}\circ f=p_{n-2}$ has an equivalent formulation in terms of the
Gromoll Filtration of the group of exotic $n$--spheres. We review the
details of this alternative formulation at the beginning of Section \ref%
{purse stab section}.
\end{remark}

Section 2 introduces notations and conventions. Section 3 is a review of
necessary tools from Alexandrov geometry, and Theorem \ref{Purse Stab} is
proven in Section 4.

Throughout the remainder of the paper, we assume, without loss of
generality, by rescaling if necessary, that $k=-1,0$ or $1$.

\begin{acknowledgment}
We are grateful to Stefano Vidussi for several conversations about exotic
differentiable structures on $\mathbb{R}P^{4}.$

We are grateful to the referees of this paper for making us aware of the
results in \cite{KMS}, and for valuable expository suggestions.
\end{acknowledgment}

\section{Conventions and Notations}

Recall that an Alexandrov space is a complete, locally compact, intrinsic
metric space with a lower curvature bound in the triangle comparison sense.
We will assume a basic familiarity with Alexandrov spaces, including, but
not limited to, \cite{BGP}. We list here several conventions that will be
used freely throughout.

Let $X$ be an $n$--dimensional Alexandrov space and $x,p,y\in X$. We call
minimal geodesics in $X$ \emph{segments} and denote by $px$ a segment in $X$
with endpoints $p$ and $x$. We let $\Sigma _{p}$ and $T_{p}X$ denote the
space of directions and tangent cone at $p$, respectively. For a geodesic
direction $v\in T_{p}X,$ we let $\gamma _{v}$ be the segment whose initial
direction is $v.$ Following \cite{Pet}, we let $\Uparrow _{x}^{p}\subset
\Sigma _{x}$ denote the set of directions of segments from $x$ to $p,$ and
we let $\uparrow _{x}^{p}\in $ $\Uparrow _{x}^{p}$ be the direction of a
single segment from $x$ to $p.$ We let $\sphericalangle (x,p,y)$ denote the
angle of a hinge formed by $px$ and $py$ and $\tilde{\sphericalangle}(x,p,y)$
denote the corresponding comparison angle.

Following \cite{OSY}, we let $\tau :\mathbb{R}^{k}\rightarrow \mathbb{R}_{+}$
be any function that satisfies 
\begin{equation*}
\lim_{x_{1},\ldots ,x_{k}\rightarrow 0}\tau \left( x_{1},\ldots
,x_{k}\right) =0,
\end{equation*}%
and abusing notation, we let $\tau :\mathbb{R}^{k}\times \mathbb{R}%
^{n}\rightarrow \mathbb{R}$ be any function that satisfies 
\begin{equation*}
\lim_{x_{1},\ldots ,x_{k}\rightarrow 0}\tau \left( x_{1},\ldots
,x_{k}|y_{1},\ldots ,y_{n}\right) =0,
\end{equation*}%
provided $y_{1},\ldots ,y_{n}$ remain fixed. When making an estimate with a
function $\tau ,$ we implicitly assert the existence of such a function for
which the estimate holds.

For $p\in X$ and $r>0,$ we set 
\begin{equation*}
B\left( p,r\right) \equiv \left\{ \left. x\in X\text{ }\right\vert \text{ 
\textrm{dist}}\left( x,p\right) <r\right\} .
\end{equation*}

%
%
%

\section{Basic Tools From Alexandrov Geometry}

Strainers, as defined in \cite{BGP}, form the core of the calculus arguments
used to prove our main theorem. To motivate them let $\left\{ v_{i}\right\}
_{i=1}^{n}$ be an orthonormal basis for $\mathbb{R}^{n}$. Notice that the
gradients of the distance functions from the $v_{i}$s are orthonormal at $0$
and almost orthonormal in a neighborhood $N$ of $0.$ Thus the map $%
f:N\rightarrow \mathbb{R}^{n},$ 
\begin{equation*}
f(x)=\left( \mathrm{dist}\left( v_{1},x\right) ,\mathrm{dist}\left(
v_{2},x\right) ,\ldots ,\mathrm{dist}\left( v_{n},x\right) \right)
\end{equation*}%
is a bi-Lipschitz embedding with Lipschitz constants that converge to $1$ as 
$N$ gets smaller.

By exponentiating an orthonormal basis, it is easy to re-create these data
around a point in a Riemannian manifold. This plus the fact that comparison
angles are continuous leads us to the definition of strainers.

\begin{definition}
Let $X$ be an Alexandrov space. A point $x\in X$ is said to be $\left(
n,\delta ,r\right) $--strained by the strainer $\left\{ \left(
a_{i},b_{i}\right) \right\} _{i=1}^{n}\subset X\times X$ provided that for
all $i\neq j$ we have%
\begin{equation*}
\begin{array}{ll}
\widetilde{\sphericalangle }\left( a_{i},x,b_{j}\right) >\frac{\pi }{2}%
-\delta , & \widetilde{\sphericalangle }\left( a_{i},x,b_{i}\right) >\pi
-\delta , \\ 
\widetilde{\sphericalangle }\left( a_{i},x,a_{j}\right) >\frac{\pi }{2}%
-\delta , & \widetilde{\sphericalangle }\left( b_{i},x,b_{j}\right) >\frac{%
\pi }{2}-\delta ,\text{ and} \\ 
\multicolumn{2}{c}{\min_{i=1,\ldots ,n}\left\{ \mathrm{dist}%
(\{a_{i},b_{i}\},x)\right\} >r.}%
\end{array}%
\end{equation*}

We say $B\subset X$ is $(n,\delta ,r)$--strained with strainer $\{\left(
a_{i},b_{i}\right) \}_{i=1}^{n}$ provided every point $x\in B$ is $(n,\delta
,r)$--strained by $\{\left( a_{i},b_{i}\right) \}_{i=1}^{n}$.
\end{definition}

The following is observed in \cite{Yam2}.

\begin{proposition}
\label{posdeltastrainedrad} Let $X$ be a compact $n$-dimensional Alexandrov
space. Then the following are equivalent:

\noindent 1. There is a (sufficiently small) $\eta >0$ so that for every $%
p\in X,$%
\begin{equation*}
\mathrm{dist}_{G-H}\left( \Sigma _{p},S^{n-1}\right) <\eta .
\end{equation*}

\noindent 2. There is a (sufficiently small) $\delta >0$ and an $r>0$ such
that $X$ is covered by finitely many $(n,\delta ,r)$--strained neighborhoods.
\end{proposition}

\begin{theorem}
\label{BGP coordinates}(\cite{BGP} Theorem 9.4) Let $X$ be an $n$%
--dimensional Alexandrov space with curvature bounded from below. Let $p\in
X $ be $\left( n,\delta ,r\right) $--strained by $\left\{ \left(
a_{i},b_{i}\right) \right\} _{i=1}^{n}.$ Provided $\delta $ is small enough,
there is a $\rho >0$ such that the map $f:B(p,\rho )\rightarrow \mathbb{R}%
^{n}$ defined by 
\begin{equation*}
f(x)=\left( \mathrm{dist}\left( a_{1},x\right) ,\mathrm{dist}\left(
a_{2},x\right) ,\ldots ,\mathrm{dist}\left( a_{n},x\right) \right)
\end{equation*}%
is a bi-Lipschitz embedding with Lipschitz constants in $\left( 1-\tau
\left( \delta ,\rho \right) ,1+\tau \left( \delta ,\rho \right) \right) .$
\end{theorem}

If $B$ is $(n,\delta ,r)$--strained by $\{a_{i},b_{i}\}_{i=1}^{n}$, any
choice of $2n$--directions, $\left\{ \left( \uparrow _{x}^{a_{i}},\uparrow
_{x}^{b_{i}}\right) \right\} _{i=1}^{n},$ where $x\in B,$ will be called a
set of straining directions for $\Sigma _{x}.$ As in, \cite{BGP} and \cite%
{Yam2}, we say an Alexandrov space $\Sigma $ with $\mathrm{curv\,}\Sigma
\geq 1$ is globally $(m,\delta )$-strained by pairs of subsets $%
\{A_{i},B_{i}\}_{i=1}^{m}$ provided 
\begin{equation*}
\begin{array}{ll}
|\mathrm{dist}(a_{i},b_{j})-\frac{\pi }{2}|<\delta , & \mathrm{dist}%
(a_{i},b_{i})>\pi -\delta , \\ 
|\mathrm{dist}(a_{i},a_{j})-\frac{\pi }{2}|<\delta , & |\mathrm{dist}%
(b_{i},b_{j})-\frac{\pi }{2}|<\delta%
\end{array}%
\end{equation*}%
for all $a_{i}\in A_{i}$, $b_{i}\in B_{i}$ and $i\neq j$.

\begin{theorem}
\label{BGP--SOY}(\cite{BGP}, Theorem 9.5, cf. also \cite{OSY}, Section 3)
Let $\Sigma $ be an $\left( n-1\right) $--dimensional Alexandrov space with
curvature $\geq 1.$ Suppose $\Sigma $ is globally strained by $%
\{A_{i},B_{i}\}$. There is a map $\tilde{\Psi}:\mathbb{R}^{n}\longrightarrow
S^{n-1}$ so that $\Psi :\Sigma \rightarrow S^{n-1}$ defined by 
\begin{equation*}
\Psi (x)=\tilde{\Psi}\circ \left( \mathrm{dist}\left( A_{1},x\right) ,%
\mathrm{dist}\left( A_{2},x\right) ,\ldots ,\mathrm{dist}\left(
A_{n},x\right) \right)
\end{equation*}%
is a bi-Lipschitz homeomorphism with Lipschitz constants in $\left( 1-\tau
\left( \delta \right) ,1+\tau \left( \delta \right) \right) $.
\end{theorem}

\begin{remark}
The description of $\tilde{\Psi}:\mathbb{R}^{n}\longrightarrow S^{n-1}$ in 
\cite{BGP} is explicit but is geometric rather than via a formula. Combining
the proof in \cite{BGP} with a limiting argument, one can see that the map $%
\Psi $ can be given by 
\begin{equation*}
\Psi (x)=\left( \sum \cos ^{2}\left( \mathrm{dist}\left( A_{i},x\right)
\right) \right) ^{-1/2}\left( \cos \left( \mathrm{dist}\left( A_{1},x\right)
\right) ,\ldots ,\cos \left( \mathrm{dist}\left( A_{n},x\right) \right)
\right) .
\end{equation*}
\end{remark}

Next we state a powerful lemma showing that for a $(1,\delta ,r)$--strained
neighborhood, angle and comparison angle almost coincide for geodesic hinges
with one side in the neighborhood and the other reaching a strainer.

\begin{lemma}
(\cite{BGP}, Lemma $5.6$) Let $B\subset X$ be $\left( 1,\delta ,r\right) $%
--strained by $(y_{1},y_{2}).$ For any $x,z\in B,$%
\begin{equation*}
\left\vert \tilde{\sphericalangle}\left( y_{1},x,z\right) +\tilde{%
\sphericalangle}\left( y_{2},x,z\right) -\pi \right\vert <\tau \left( \delta
,\mathrm{dist}\left( x,z\right) |r\right) .
\end{equation*}%
In particular, for $i=1,2$, 
\begin{equation*}
\left\vert \sphericalangle \left( y_{i},x,z\right) -\tilde{\sphericalangle}%
\left( y_{i},x,z\right) \right\vert <\tau \left( \delta ,\mathrm{dist}\left(
x,z\right) |r\right) .
\end{equation*}
\end{lemma}

\begin{corollary}
\label{Angle continuity}Let $B\subset X$ be $\left( 1,\delta ,r\right) $%
--strained by $\left( a,b\right) $. Let $\left\{ X^{\alpha }\right\}
_{\alpha =1}^{\infty }$ be a sequence of Alexandrov spaces with $\mathrm{curv%
}X^{\alpha }\geq k$ such that $X^{\alpha }\longrightarrow X.$ For $x,z\in B$%
, suppose that $a^{\alpha },b^{\alpha },x^{\alpha },z^{\alpha }\in X^{\alpha
}$ converge to $a,b,x,$ and $z,$ respectively. Then 
\begin{equation*}
\left\vert \sphericalangle \left( a^{\alpha },x^{\alpha },z^{\alpha }\right)
-\sphericalangle \left( a,x,z\right) \right\vert <\tau \left( \delta ,\frac{1%
}{\alpha },\mathrm{dist}\left( x,z\right) \text{ }|\text{ }r\right) .
\end{equation*}
\end{corollary}

\begin{proof}
The convergence $X^{\alpha }\longrightarrow X$ implies that%
\begin{equation*}
\left\vert \tilde{\sphericalangle}\left( a^{\alpha },x^{\alpha },z^{\alpha
}\right) -\tilde{\sphericalangle}\left( a,x,z\right) \right\vert <\tau
\left( \frac{1}{\alpha }\text{ }|\text{ }\mathrm{dist}\left( x,z\right)
,r\right) .
\end{equation*}%
Combined with the previous lemma,%
\begin{eqnarray*}
\left\vert \sphericalangle \left( a^{\alpha },x^{\alpha },z^{\alpha }\right)
-\sphericalangle \left( a,x,z\right) \right\vert &\leq &\left\vert
\sphericalangle \left( a^{\alpha },x^{\alpha },z^{\alpha }\right) -\tilde{%
\sphericalangle}\left( a^{\alpha },x^{\alpha },z^{\alpha }\right)
\right\vert + \\
&&\left\vert \tilde{\sphericalangle}\left( a^{\alpha },x^{\alpha },z^{\alpha
}\right) -\tilde{\sphericalangle}\left( a,x,z\right) \right\vert +\left\vert 
\tilde{\sphericalangle}\left( a,x,z\right) -\sphericalangle \left(
a,x,z\right) \right\vert \\
&\leq &2\tau \left( \delta ,\frac{1}{\alpha },\mathrm{dist}\left( x,z\right)
|r\right) +\tau \left( \frac{1}{\alpha }\text{ }|\text{ }\mathrm{dist}\left(
x,z\right) ,r\right) \\
&=&\tau \left( \delta ,\frac{1}{\alpha },\mathrm{dist}\left( x,z\right) 
\text{ }|\text{ }r\right) .
\end{eqnarray*}
\end{proof}

\begin{lemma}
\label{angle convergence}Let $B\subset X$ be $(n,\delta ,r)$--strained by $%
\left\{ \left( a_{i},b_{i}\right) \right\} _{i=1}^{n}$. Let $\left\{
X^{\alpha }\right\} _{\alpha =1}^{\infty }$ have $\mathrm{curv}X^{\alpha
}\geq k,$ and suppose that $X_{\alpha }\longrightarrow $ $X$. Let $\left\{
\left( \gamma _{1,\alpha },\gamma _{2,\alpha }\right) \right\} _{\alpha
=1}^{\infty }$ be a sequence of geodesic hinges in the $X^{\alpha }$ that
converge to a geodesic hinge $\left( \gamma _{1},\gamma _{2}\right) $ with
vertex in $B.$ Then 
\begin{equation*}
\left\vert \sphericalangle \left( \gamma _{1,\alpha }^{\prime }\left(
0\right) ,\gamma _{2,\alpha }^{\prime }\left( 0\right) \right)
-\sphericalangle \left( \gamma _{1}^{\prime }\left( 0\right) ,\gamma
_{2}^{\prime }\left( 0\right) \right) \right\vert <\tau \left( \delta
,1/\alpha \text{ }|\text{ }l\left( \gamma _{1}\right) ,l\left( \gamma
_{2}\right) ,r\right) ,
\end{equation*}%
where $l\left( \gamma _{i}\right) $ is the length of $\gamma _{i}.$
\end{lemma}

\begin{remark}
Note that without the strainer, $\lim \inf_{\alpha \rightarrow \infty
}\sphericalangle \left( \gamma _{1,\alpha }^{\prime }\left( 0\right) ,\gamma
_{2,\alpha }^{\prime }\left( 0\right) \right) \geq \sphericalangle \left(
\gamma _{1}^{\prime }\left( 0\right) ,\gamma _{2}^{\prime }\left( 0\right)
\right) $ \cite{GrovPet2, BGP}.
\end{remark}

\begin{proof}
Apply the previous corollary with $x^{\alpha }=\gamma _{1,\alpha }\left(
0\right) ,$ $z^{\alpha }=\gamma _{1,\alpha }\left( \varepsilon \right) ,$ $%
x^{\alpha }\rightarrow x,$ and $z^{\alpha }\rightarrow z$ to conclude 
\begin{equation*}
\left\vert \sphericalangle (\Uparrow _{x^{\alpha }}^{a_{i}^{\alpha }},\gamma
_{1,\alpha }^{\prime }\left( 0\right) )-\sphericalangle (\Uparrow
_{x}^{a_{i}},\gamma _{1}^{\prime }\left( 0\right) )\right\vert <\tau \left(
\delta ,\frac{1}{\alpha },\mathrm{dist}\left( x,z\right) \text{ }|\text{ }%
r\right) .
\end{equation*}%
Similar reasoning with $x^{\alpha }=\gamma _{2,\alpha }\left( 0\right) ,$ $%
z^{\alpha }=\gamma _{2,\alpha }\left( \varepsilon \right) ,$ $x=\lim_{\alpha
\rightarrow \infty }x^{\alpha },$ and $z=\lim_{\alpha \rightarrow \infty
}z^{\alpha }$ gives 
\begin{equation*}
\left\vert \sphericalangle (\Uparrow _{x^{\alpha }}^{a_{i}^{\alpha }},\gamma
_{2,\alpha }^{\prime }\left( 0\right) )-\sphericalangle (\Uparrow
_{x}^{a_{i}},\gamma _{2}^{\prime }\left( 0\right) )\right\vert <\tau \left(
\delta ,\frac{1}{\alpha },\mathrm{dist}\left( x,z\right) \text{ }|\text{ }%
r\right) .
\end{equation*}

Since $\mathrm{dist}\left( x,z\right) $ may be as small as we please, the
result then follows from Theorem \ref{BGP--SOY}.
\end{proof}

%
%

\section{Purse Stability\label{purse stab section}}

We start this section with a review of Gromoll groups. We then state Theorem %
\ref{Gromoll} and show that it implies Theorem \ref{Purse Stab}. The bulk of
this section is devoted to the proof of Theorem \ref{Gromoll}.

Recall that a twisted $n$--sphere, $\Sigma ^{n},$ is a compact smooth
manifold that admits a Morse function $f$ with exactly two critical points.
The gradient flow of $f$ allows us to decompose $\Sigma ^{n}$ as the union
of two $n$--disks. In \cite{KerMil}, Kervaire and Milnor showed that the
twisted $n$--spheres form a group $\Gamma ^{n}$ under connected sum. Gromoll
showed that there is a filtration 
\begin{equation*}
\left\{ e\right\} \subset \Gamma _{n-1}^{n}\subset \cdots \subset \Gamma
_{1}^{n}=\Gamma ^{n}
\end{equation*}%
by subgroups, which are now called Gromoll groups \cite{Grom}. Rather than
using the definition of the $\Gamma _{q}^{n}$s from \cite{Grom}, we use the
equivalent notion from Theorem D in \cite{GrovWilh2}.

\begin{definition}
Let 
\begin{equation*}
f:S^{q-1}\times S^{n-q}\longrightarrow S^{q-1}\times S^{n-q}
\end{equation*}%
be a diffeomorphism that satisfies%
\begin{equation}
p_{q-1}\circ f=p_{q-1},  \label{gluing eqn}
\end{equation}%
where 
\begin{equation*}
p_{q-1}:S^{q-1}\times S^{n-q}\longrightarrow S^{q-1}
\end{equation*}%
is projection to the first factor. Then $\Gamma _{q}^{n}$ consists of those
smooth manifolds that are diffeomorphic to 
\begin{equation}
D^{q}\times S^{n-q}\cup _{f}S^{q-1}\times D^{n-q+1}.  \label{handle body}
\end{equation}
\end{definition}

\begin{theorem}
\label{Gromoll}Let $\left\{ M^{\alpha }\right\} _{\alpha =1}^{\infty }$ be a
sequence of closed, Riemannian $n$--manifolds with 
\begin{equation*}
\mathrm{sec\,}M^{\alpha }\geq k
\end{equation*}%
so that 
\begin{equation*}
M_{\alpha }\longrightarrow P_{k,r}^{n}
\end{equation*}%
in the Gromov-Hausdorff topology. Then for $\alpha $ sufficiently large, $%
M_{\alpha }\in \Gamma _{n-1}^{n}.$
\end{theorem}

It is known that $\Gamma _{n-1}^{n}$ is trivial for all $n.$ Given this
fact, Theorem \ref{Purse Stab} implies Theorem \ref{Gromoll}. To see why $%
\Gamma _{n-1}^{n}$ is trivial, we first point out that $\Gamma ^{n}=\left\{
e\right\} $ for $n=1,2,3,$ \cite{Munk}. So we may assume that $n\geq 4.$
Next notice that a diffeomorphism $f:S^{n-2}\times S^{1}\longrightarrow
S^{n-2}\times S^{1}$ so that $p_{n-2}\circ f=p_{n-2}$ gives rise to an
element of $\pi _{n-2}\left( \mathrm{Diff}_{+}\left( S^{1}\right) \right) .$
If two such diffeomorphisms give the same homotopy class, then the
construction (\ref{handle body}) yields diffeomorphic manifolds (cf. \cite%
{GrovWilh2}). Since the group of orientation preserving diffeomorphisms of
the circle deformation retracts to $SO\left( 2\right) ,$ it follows that for 
$n\geq 4,$ $\Gamma _{n-1}^{n}=\left\{ e\right\} ,$ as desired.

\subsection{The Model Submetry}

View $P_{k,r}^{n}$ as the double of the half disk $\left\{ \frac{1}{2}%
\mathcal{D}_{k}^{n}\left( r\right) \right\} ^{+},$%
\begin{equation*}
P_{k,r}^{n}\equiv \mathrm{Double}\left( \left\{ \frac{1}{2}\mathcal{D}%
_{k}^{n}\left( r\right) \right\} ^{+}\right) ,
\end{equation*}%
and let $\left\{ M^{\alpha }\right\} _{\alpha =1}^{\infty }$ be a sequence
of closed, Riemannian $n$--manifolds with 
\begin{equation*}
\mathrm{sec\,}M^{\alpha }\geq k
\end{equation*}%
and 
\begin{equation*}
\mathrm{dist}_{GH}\left( M_{\alpha },P_{k,r}^{n}\right) <\frac{1}{\alpha }.
\end{equation*}%
Our Model Submetry 
\begin{equation}
\Psi :P_{k,r}^{n}\longrightarrow \mathbb{R}^{n-1}  \label{dfn of psi}
\end{equation}%
is the restriction to either half disk of orthogonal projection to the
totally geodesic hyperplane $H\subset \mathcal{D}_{k}^{n}\left( r\right) $
that defines $P_{k,r}^{n}.$

In this subsection, we describe the Model Submetry in terms of distance
functions on the Purse. This will enable us, in the next subsection, to
approximate $\Psi $ by maps $\Psi _{\alpha }:M_{a}\longrightarrow \mathbb{R}%
^{n-1}$ that inherit much of the regularity of $\Psi .$ The inherited
regularity is established in the paper's final subsection in Corollary \ref%
{cor 6} and Lemma \ref{Pis on disk bundle lemma}. It will allow us to
decompose $M_{a}$ as the union of a trivial $D^{2}$--bundle and a trivial
circle bundle (see the circle and disk bundle lemmas, \ref{circle bundle},%
\ref{Disk Bundle}, below). The proof of Theorem \ref{Gromoll} is completed
by showing that $M_{\alpha }$ is the union of these two bundles glued
together on their common boundary via a diffeomorphism that satisfies
Equation (\ref{gluing eqn}).

To describe $\Psi $ in terms of distance functions we use 
\begin{equation*}
H^{n}\equiv \left\{ \left. (x_{0},x_{1},\cdots ,x_{n})\in \mathbb{R}%
^{n+1}\right\vert -\left( x_{0}\right) ^{2}+\left( x_{1}\right) ^{2}+\cdots
+\left( x_{n}\right) ^{2}=-1,\text{ }x_{0}>0\right\}
\end{equation*}%
as our model for hyperbolic space. We write $\mathcal{S}_{k}^{n}$ for any of 
$H^{n}\subset \mathbb{R}^{n+1},$ $\left\{ e_{0}\right\} \times \mathbb{R}%
^{n}\subset \mathbb{R}^{n+1},$ or $S^{n}\subset \mathbb{R}^{n+1}.$ We denote
the standard basis for $\mathbb{R}^{n+1}$ by $\left\{ e_{0},e_{1},\ldots
,e_{n}\right\} $, and we identify $\mathcal{D}_{k}^{n}\left( r\right) $ with 
\begin{equation*}
\mathcal{D}_{k}^{n}\left( r\right) \equiv \left\{ 
\begin{array}{ll}
\left\{ \left. z\in H^{n}\subset \mathbb{R}^{n+1}\right\vert \mathrm{dist}%
_{H^{n}}\left( e_{0},z\right) \leq r\right\} & \text{if }k=-1 \\ 
\left\{ \left. z\in \left\{ e_{0}\right\} \times \mathbb{R}^{n}\subset 
\mathbb{R}^{n+1}\right\vert \mathrm{dist}_{\mathbb{R}^{n+1}}\left(
e_{0},z\right) \leq r\right\} & \text{if }k=0 \\ 
\left\{ \left. z\in S^{n}\subset \mathbb{R}^{n+1}\right\vert \mathrm{dist}%
_{S^{n}}\left( e_{0},z\right) \leq r\right\} & \text{if }k=1.%
\end{array}%
\right.
\end{equation*}

Set%
\begin{equation*}
p_{0}=e_{0},
\end{equation*}%
and for $i\in \left\{ 1,2,\ldots ,n-1\right\} ,$ set 
\begin{equation}
p_{i}\equiv \left\{ 
\begin{array}{ll}
\cosh (r)e_{0}+\sinh (r)e_{i} & \text{if }k=-1 \\ 
e_{0}+re_{i} & \text{if }k=0 \\ 
\cos (r)e_{0}-\sin (r)e_{i} & \text{if }k=1.%
\end{array}%
\right.  \label{dfn of  p_i dfn}
\end{equation}

We let the totally geodesic hyperplane $H\subset \mathcal{D}_{k}^{n}\left(
r\right) $ that defines $P_{k,r}^{n}$ be the one containing $%
p_{0},p_{1},\ldots ,p_{n-1}.$ We denote the singular subset of $P_{k,r}^{n}$
by $\mathcal{S}^{n-2},$ that is, $\mathcal{S}^{n-2}$ is the copy of the $%
\left( n-2\right) $-sphere which is the boundary of the $(n-1)$--disk $%
\mathcal{D}_{k}^{n}(r)\cap H.$ Thus $\left\{ p_{i}\right\}
_{i=1}^{n-1}\subset $ $\mathcal{S}^{n-2}$.

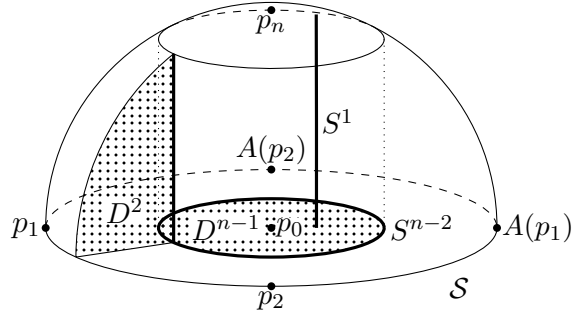
\begin{figure}[ht]
\centering
\begin{tikzpicture}[ dot/.style={fill=black,circle,minimum size=3pt}] 

\def\R{3} 
\def\angEl{15} 
\def\diskloc{210} 
\draw (0,0) (0:\R) arc (0:180:\R);
\DrawLatitudeCircle[\R]{0} 
\DrawLatitudeCirclefilledD[{\R*.5}]{0}
\DrawLatitudeCircle[\R]{60}

\DrawLongitudeCircleportion[\R]{\diskloc} 
\Drawcircleoverdisk[\R]{0}{.2}
\Drawlightcircleoverdisk[\R]{0}{\R/2}
\Drawlightcircleoverdisk[\R]{180}{\R/2}
\draw (2.5,-.8) node{\small $\mathcal S$};
 \draw (2,0) node{\small$S^{n-2}$};
\Drawlabels[\R]{0}
\end{tikzpicture}
\caption{One side of $P^n_{k,r}$ for $n=3$ and $k=0$.}\label{pursefig}
\end{figure}
Since the antipodal map $A:\mathcal{D}_{k}^{n}\left( r\right)
\longrightarrow \mathcal{D}_{k}^{n}\left( r\right) $ commutes with the
reflection $R$ in $H,$ it induces a well-defined involution of $%
A_{P}:P_{k,r}^{n}\longrightarrow P_{k,r}^{n}.$ Note that $A_{P}$ restricts
to the antipodal map of $\mathcal{S}^{n-2}$ and fixes the circle at maximal
distance from $\mathcal{S}^{n-2}.$ For simplified notation, we will write $A$
for the restriction of $A_{P}$ to $\mathcal{S}^{n-2}.$

For $i\in \left\{ 1,2,\ldots ,n-1\right\} ,$ set 
\begin{equation*}
f_{i}(x)\equiv h_{k}\circ \mathrm{dist}\left( A\left( p_{i}\right) ,x\right)
-h_{k}\circ \mathrm{dist}\left( p_{i},x\right)
\end{equation*}%
where $h_{k}:\mathbb{R}\rightarrow \mathbb{R}$ is defined as 
\begin{equation*}
h_{k}(x)\equiv \left\{ 
\begin{array}{ll}
\frac{1}{2\sinh r}\cosh (x) & \text{if }k=-1\vspace*{0.05in} \\ 
\frac{x^{2}}{4r} & \text{if }k=0 \\ 
\frac{1}{2\sin r}\cos (x) & \text{if }k=1.%
\end{array}%
\right.
\end{equation*}%
The functions $\{f_{i}\}_{i=1}^{n-1}$ are then restrictions of $\left(
n-1\right) $--coordinate functions of $\mathbb{R}^{n+1}$ to $\mathcal{D}%
_{k}^{n}\left( r\right) \subset \mathcal{S}_{k}^{n}.$ In particular, the $%
f_{i}$s are the coordinate functions of the Model Submetry $\Psi
:P_{k,r}^{n}\longrightarrow \mathbb{R}^{n-1}$ from (\ref{dfn of psi}), that
is, 
\begin{equation*}
\Psi =\left( f_{1},f_{2},\ldots ,f_{n-1}\right) .
\end{equation*}%
It follows that $\Psi |_{\mathcal{S}^{n-2}}$ is the the inclusion of $%
\mathcal{S}^{n-2}$ into $\mathcal{S}_{k}^{n}\subset \mathbb{R}^{n+1}$.

To construct our decompositions of the $M_{\alpha }$s into $D^{n-1}\times
S^{1}$ and $S^{n-2}\times D^{2},$ we next approximate the Model Submetry by
maps $\Psi _{d}^{\alpha }:M^{\alpha }\longrightarrow \mathbb{R}^{n-1}.$

\subsection{Approximating The Model Submetry}

Let $\left\{ M^{\alpha }\right\} _{\alpha =1}^{\infty }$ be a sequence of
closed, Riemannian $n$--manifolds with 
\begin{equation*}
\mathrm{sec\,}M^{\alpha }\geq k
\end{equation*}%
and 
\begin{equation*}
\mathrm{dist}_{GH}\left( M_{\alpha },P_{k,r}^{n}\right) <\frac{1}{\alpha }.
\end{equation*}

Let $A:M^{\alpha }\longrightarrow M^{\alpha }$ denote any map that is
Gromov-Hausdorff close to $A:P_{k,r}^{n}\longrightarrow P_{k,r}^{n}.$

We define approximations $f_{i,d}^{\alpha }:M^{\alpha }\longrightarrow 
\mathbb{R}$ of the $f_{i}$s by 
\begin{eqnarray*}
f_{i,d}^{\alpha }(x) &=&\frac{1}{\mathrm{vol}\left( B\left( A\left(
p_{i}^{\alpha }\right) ,d\right) \right) }\int_{z\in B\left( A\left(
p_{i}^{\alpha }\right) ,d\right) }h_{k}\circ \mathrm{dist}\left( z,x\right)
\\
&&\hspace*{1.9in}-\frac{1}{\mathrm{vol}\left( B\left( p_{i}^{\alpha
},d\right) \right) }\int_{z\in B\left( p_{i}^{\alpha },d\right) }h_{k}\circ 
\mathrm{dist}\left( z,x\right) .
\end{eqnarray*}

We let $\Psi _{d}^{\alpha }:M^{\alpha }\longrightarrow \mathbb{R}^{n-1}$ be
defined by%
\begin{equation*}
\Psi _{d}^{\alpha }=(f_{1,d}^{\alpha },\dots ,f_{n-1,d}^{\alpha }).
\end{equation*}

\subsection{The Bundle Decomposition\label{bundle sect}}

To prove Theorem \ref{Gromoll} we decompose the $M_{\alpha }$s as the union
of a trivial circle bundle, $D^{n-1}\times S^{1},$ and a trivial disk
bundle, $S^{n-2}\times D^{2},$ which we describe in this subsection.

We identify $\mathbb{R}^{n-1}$ with 
\begin{equation*}
\mathbb{R}^{n-1}\equiv \mathrm{span}\left\{ e_{1},\ldots ,e_{n-1}\right\} .
\end{equation*}%
For small $\varepsilon >0$, we set 
\begin{eqnarray*}
E_{D}\left( r-\varepsilon \right) &\equiv &\left( \Psi \right)
^{-1}(D^{n-1}(0,r-\varepsilon )), \\
E_{D}^{\alpha }\left( r-\varepsilon \right) &\equiv &\left( \Psi
_{d}^{\alpha }\right) ^{-1}(D^{n-1}(0,r-\varepsilon )), \\
E_{A}\left( r-\varepsilon \right) &\equiv &\left( \Psi \right) ^{-1}(%
\overline{A^{n-1}(0,r-\varepsilon ,2r)}),\text{ and} \\
E_{A}^{\alpha }\left( r-\varepsilon \right) &\equiv &\left( \Psi
_{d}^{\alpha }\right) ^{-1}(\overline{A^{n-1}(0,r-\varepsilon ,2r)}),
\end{eqnarray*}%
where $\overline{A^{n-1}(0,r-\varepsilon ,2r)}$ is the closed annulus in $%
\mathbb{R}^{n-1}$ centered at $0$ with inner radius $r-\varepsilon $ and
outer radius $2r,$ and $D^{n-1}(0,r-\varepsilon )$ is the closed ball in $%
\mathbb{R}^{n-1}$ centered at $0$ with radius $r-\varepsilon .$

The next two Lemmas give us the desired bundle decomposition of the $%
M_{\alpha }$s. Hence together they imply Theorem \ref{Gromoll}.

\begin{circlebundlelemma}
\label{circle bundle}For any sufficiently small $\varepsilon >0$, 
\begin{equation*}
\Psi _{d}^{\alpha }:E_{D}^{\alpha }\left( r-\varepsilon \right)
\longrightarrow D^{n-1}(0,r-\varepsilon )
\end{equation*}%
is a trivial $S^{1}$--bundle, provided $\alpha $ is sufficiently large and $%
d $ is sufficiently small.
\end{circlebundlelemma}

Let $\mathrm{pr}:\overline{A^{n-1}(0,r-\varepsilon ,2r)}\rightarrow \partial
\left( D^{n-1}(0,r-\varepsilon )\right) =S^{n-2}$ be radial projection and
set 
\begin{eqnarray*}
g\equiv &&\mathrm{pr}\circ \Psi :E_{A}\left( r-\varepsilon \right)
\rightarrow \partial \left( D^{n-1}(0,r-\varepsilon )\right) \\
g_{d}^{\alpha }\equiv &&\mathrm{pr}\circ \Psi _{d}^{\alpha }:E_{A}^{\alpha
}\left( r-\varepsilon \right) \rightarrow \partial \left(
D^{n-1}(0,r-\varepsilon )\right) .
\end{eqnarray*}

\begin{diskbundlelemma}
\label{Disk Bundle}There is an $\varepsilon >0$ so that 
\begin{equation*}
g_{d}^{\alpha }:E_{A}^{\alpha }\left( r-\varepsilon \right) \longrightarrow
\partial \left( D^{n-1}(0,r-\varepsilon )\right)
\end{equation*}%
is a trivial $D^{2}$--bundle over $\partial \left( D^{n-1}(0,r-\varepsilon
)\right) =S^{n-2}$, provided $\alpha $ is sufficiently large and $d$ is
sufficiently small.
\end{diskbundlelemma}

\begin{proof}[Proof of Theorem \protect\ref{Gromoll} assuming the circle and
disk bundle lemmas]
To simplify notation, set $D^{n-1}=D^{n-1}(0,r-\varepsilon )$ and $\partial
D^{n-1}=S^{n-2}=\partial D^{n-1}(0,r-\varepsilon ).$ Let 
\begin{eqnarray}
\begin{tikzpicture}[description/.style={fill=white,inner sep=2pt}]
\matrix (m) [matrix of math nodes, row sep=3em,
column sep=2.5em, text height=1.5ex, text depth=0.25ex]
{ E_{D}^{\alpha }\left( r-\varepsilon \right) & & D^{n-1}\times S^{1} \\
& S^{n-2} & \\ };
\path[->,font=\scriptsize]
(m-1-1) edge node[auto] {$ \Phi _{D} $} (m-1-3)
edge node[auto,swap] {$  \Psi _{d}^{\alpha } $} (m-2-2)
(m-1-3) edge node[auto] {$ p_{1} $} (m-2-2);
\end{tikzpicture}
\label{D trivial diag}
\end{eqnarray}and%
\begin{equation}
\begin{tikzpicture}[description/.style={fill=white,inner sep=2pt}]
\matrix (m) [matrix of math nodes, row sep=3em,
column sep=2.5em, text height=1.5ex, text depth=0.25ex]
{ E_{A}^{\alpha }\left( r-\varepsilon \right) & & S^{n-2}\times D^{2} \\
& S^{n-2} & \\ };
\path[->,font=\scriptsize]
(m-1-1) edge node[auto] {$ \Phi _{A} $} (m-1-3)
edge node[auto,swap] {$  g_{d}^{\alpha} $} (m-2-2)
(m-1-3) edge node[auto] {$ p_{1} $} (m-2-2);
\end{tikzpicture}
\label{A triv diag}
\end{equation}%

\begin{center}

\end{center}
be trivializations of $\Psi _{d}^{\alpha }$ and $g_{d}^{\alpha }.$

By the circle and disk bundle lemmas, 
\begin{equation*}
M_{\alpha }=E_{D}^{\alpha }\left( r-\varepsilon \right) \cup _{\Phi
_{A}\circ \Phi _{D}|_{\partial \left( D^{n-1}\times S^{1}\right)
}^{-1}}E_{A}^{\alpha }\left( r-\varepsilon \right)
\end{equation*}%
with $E_{D}^{\alpha }\left( r-\varepsilon \right) \cong D^{n-1}\times S^{1}$%
, $E_{A}^{\alpha }\left( r-\varepsilon \right) \cong S^{n-2}\times D^{2},$
and $E_{D}^{\alpha }\left( r-\varepsilon \right) \cap E_{A}^{\alpha }\left(
r-\varepsilon \right) \cong S^{n-2}\times S^{1}.$ So we only need to verify
that the gluing map satisfies 
\begin{equation*}
p_{1}\circ \Phi _{A}\circ \Phi _{D}|_{\partial \left( D^{n-1}\times
S^{1}\right) }^{-1}=p_{1},
\end{equation*}%
where $p_{1}:S^{n-2}\times S^{1}\longrightarrow S^{n-2}$ is projection onto
the first factor.

Observe that%
\begin{eqnarray*}
g_{d}^{\alpha }|_{\partial \left( E_{D}^{\alpha }\left( r-\varepsilon
\right) \right) } &=&\mathrm{pr}\circ \Psi _{d}^{\alpha }|_{\partial \left(
E_{D}^{\alpha }\left( r-\varepsilon \right) \right) },\text{ by the
definition of }g_{d}^{\alpha } \\
&=&\Psi _{d}^{\alpha }|_{\partial \left( E_{A}^{\alpha }\left( r-\varepsilon
\right) \right) },
\end{eqnarray*}%
since $\mathrm{pr}$ is a retraction onto $\Psi _{d}^{\alpha }\left( \partial
\left( E_{A}^{\alpha }\left( r-\varepsilon \right) \right) \right) =\partial
D^{n-1}(0,r-\varepsilon ).$%
\begin{equation}
g_{d}^{\alpha }|_{\partial \left( E_{D}^{\alpha }\left( r-\varepsilon
\right) \right) }=\Psi _{d}^{\alpha }|_{\partial \left( E_{A}^{\alpha
}\left( r-\varepsilon \right) \right) }.  \label{g=psi eqn}
\end{equation}%
Thus%
\begin{eqnarray*}
p_{1}\circ \Phi _{A}\circ \Phi _{D}|_{\partial \left( D^{n-1}\times
S^{1}\right) }^{-1} &=&g_{d}^{\alpha }\circ \Phi _{D}|_{\partial \left(
D^{n-1}\times S^{1}\right) }^{-1},\text{ by \ref{A triv diag}} \\
&=&\Psi _{d}^{\alpha }\circ \Phi _{D}|_{\partial \left( D^{n-1}\times
S^{1}\right) }^{-1},\text{ by \ref{g=psi eqn}} \\
&=&p_{1},\text{ by \ref{D trivial diag},}
\end{eqnarray*}%
as desired.
\end{proof}

Before proving the Circle and Disk Bundle Lemmas we establish some
preliminary machinery.

Since every space of directions of $P_{k,r}^{n}$ contains an isometrically
embedded, totally geodesic copy of $S^{n-3},$ and every space of directions
of $P_{k,r}^{n}\setminus \mathcal{S}^{n-2}$ is isometric to $S^{n-1},$ we
get the following (cf Proposition \ref{posdeltastrainedrad}).

\begin{proposition}
\label{Ufm str rad Purse}There are $r,\delta >0$ so that every point in the
purse $P_{k,r}^{n}$ has a neighborhood $B$ that is $\left( n-2,\delta
,r\right) $--strained.

For any neighborhood $U$ of $\mathcal{S}^{n-2},$ there are $r,\delta >0$ so
that every point in $P_{k,r}^{n}\setminus U$ has a neighborhood $B$ that is $%
\left( n,\delta ,r\right) $--strained.
\end{proposition}

\begin{remark}
For $x\in \mathcal{S}^{n-2},$ the strainer $\left\{ \left(
a_{i},b_{i}\right) \right\} _{i=1}^{n-2}$ can be chosen to lie in $\mathcal{S%
}^{n-2}.$
\end{remark}

Because the $f_{i}:P_{k,r}^{n}\longrightarrow \mathbb{R}$ are coordinate
functions, $\Psi |_{\mathcal{D}_{k}^{n}(r)\cap H}$ differs from the identity
by translation by $e_{0}.$ Using this we prove the following.

\begin{proposition}
\label{g submersion}There is a neighborhood $U$ of $\mathcal{S}^{n-2}\subset
P_{k,r}^{n}$ so that for any family of open sets $U^{\alpha }\subset
M^{\alpha }$ with $U^{\alpha }\rightarrow U,$ $g_{d}^{\alpha }|_{U^{\alpha
}} $ is a submersion, provided $\alpha $ is sufficiently large and $d$ is
sufficiently small.
\end{proposition}

We will show that our bundle lemmas hold for any $\varepsilon >0$ such that 
\begin{equation*}
\Psi ^{-1}\left( \overline{A^{n-1}(0,r-\varepsilon ,r)}\right) \subset U.
\end{equation*}

Since $\left\{ f_{i}\right\} _{i=1}^{n-1}$ are the $\left( n-1\right) $%
--coordinate functions for the standard embedding of $\mathcal{S}%
^{n-2}\subset \mathbb{R}^{n-1}+e_{0}$, we have

\begin{lemma}
\label{Psi sbmersion}There is a $\lambda >0$ so that for all $v\in T\mathcal{%
S}^{n-2},$ there is a $j$ so that the $j^{th}$--component function of $g$
satisfies%
\begin{equation}
\left\vert D_{v}\left( g_{j}\right) \right\vert >\lambda \left\vert
v\right\vert .  \label{no addendums!}
\end{equation}%
Moreover, there is a $\rho >0$ so that for all $x\in B\left( p_{i},\rho
\right) \cup B\left( A\left( p_{i}\right) ,\rho \right) $ and all $v\in T_{x}%
\mathcal{S}^{n-2}$, the index $j$ in Inequality (\ref{no addendums!}) can be
chosen to be different from $i.$
\end{lemma}

To lift Lemma \ref{Psi sbmersion} to the $M^{\alpha }$s, we need an analog
of $T\mathcal{S}^{n-2}$ within each $M^{\alpha },$ or better, a notion of $%
g_{d}^{\alpha }$--almost horizontal for each $U^{\alpha }\subset M^{\alpha
}. $ To achieve this, cover $\mathcal{S}^{n-2}$ by a finite number of $%
\left( n-2,\delta ,\rho \right) $--strained neighborhoods $B\subset $ $%
P_{k,r}^{n}$ with strainers $\left\{ \left( a_{i},b_{i}\right) \right\}
_{i=1}^{n-2}\subset \mathcal{S}^{n-2}.$ Let $U$ be the union of this finite
collection, and let $U^{\alpha }\subset M^{\alpha }$ converge to $U.$

Given $x^{\alpha }\in U^{\alpha },$ we now define a $g_{d}^{\alpha }$%
--almost horizontal space at $x^{\alpha }$ as follows. Let $B^{\alpha }$ be
a $\left( n-2,\delta ,\rho \right) $--strained neighborhood for $x^{\alpha }$
with strainers $\left\{ \left( a_{i}^{\alpha },b_{i}^{\alpha }\right)
\right\} _{i=1}^{n-2}$ that converge 
\begin{equation}
\left( B^{\alpha },\left\{ \left( a_{i}^{\alpha },b_{i}^{\alpha }\right)
\right\} _{i=1}^{n-2}\right) \longrightarrow \left( B,\left\{ \left(
a_{i},b_{i}\right) \right\} _{i=1}^{n-2}\right) ,  \label{U strainers}
\end{equation}%
where $\left( B,\left\{ \left( a_{i},b_{i}\right) \right\}
_{i=1}^{n-2}\right) $ is part of our finite collection of $\left( n-2,\delta
,\rho \right) $--strained neighborhoods for points in $\mathcal{S}%
^{n-2}\subset P_{k,r}^{n}.$ We set 
\begin{equation*}
H_{x^{\alpha }}^{g_{d}^{\alpha }}\equiv \mathrm{span}_{i\in \left\{ 1,\ldots
,n-2\right\} }\left\{ \uparrow _{x^{\alpha }}^{a_{i}^{\alpha }}\right\} ,
\end{equation*}%
where $\uparrow _{x^{\alpha }}^{a_{i}^{\alpha }}$ is the direction of \emph{%
some }segment from $x^{\alpha }$ back to $a_{i}^{\alpha }.$ The definition
of $H_{x^{\alpha }}^{g_{d}^{\alpha }}$ depends on the choice of
neighborhood, the choice of strainers, and the choice of the directions $%
\uparrow _{x^{\alpha }}^{a_{i}^{\alpha }}.$

Regardless of these choices, $H_{x^{\alpha }}^{g_{d}^{\alpha }}$ satisfies
the following Lemma, which follows from Corollary \ref{Angle continuity}.

\begin{lemma}
\label{close deriv lemma}Let $\left\{ \left( a_{i},b_{i}\right) \right\}
_{i=1}^{n-2}$ and $\left\{ \left( a_{i}^{\alpha },b_{i}^{\alpha }\right)
\right\} _{i=1}^{n-2}$ be as in (\ref{U strainers}). For $\rho >0,$ 
\begin{equation*}
x\in U\setminus \left\{ B\left( p_{j},\rho \right) \cup B\left( A\left(
p_{j}\right) ,\rho \right) \right\} ,
\end{equation*}%
and $x^{\alpha }\in U^{\alpha }\setminus \left\{ B\left( p_{j}^{\alpha
},\rho \right) \cup B\left( A\left( p_{j}^{\alpha }\right) ,\rho \right)
\right\} $ with $\mathrm{dist}\left( x^{\alpha },x\right) <\frac{1}{\alpha }%
, $ we have 
\begin{equation*}
\left\vert D_{\uparrow _{x^{\alpha }}^{a_{i}^{\alpha }}}f_{j,d}^{\alpha
}-D_{\uparrow _{x}^{a_{i}}}f_{j}\right\vert <\tau \left( \left. \delta ,%
\frac{1}{\alpha },d\right\vert \rho \right) .
\end{equation*}
\end{lemma}

The following is a corollary of Lemma \ref{Psi sbmersion}.

\begin{corollary}
\label{expand to U cor}Let $\left( B,\left\{ \left( a_{i},b_{i}\right)
\right\} _{i=1}^{n-2}\right) \subset U$ be as in (\ref{U strainers}). There
are $\lambda ,\varepsilon >0$ so that for $x\in B$ and $v\in \Sigma _{x}$
with 
\begin{equation}
\dsum\limits_{i=1}^{n-2}\cos \sphericalangle \left( v,\uparrow
_{x}^{a_{i}}\right) >1-\varepsilon ,  \label{alom horiz for  S}
\end{equation}%
there is a $j$ so that the $j^{th}$--component function of $g$ satisfies%
\begin{equation}
\left\vert D_{v}\left( g_{j}\right) \right\vert >\lambda \left\vert
v\right\vert ,  \label{der g is big}
\end{equation}%
provided $\mathrm{diam}\left( B\right) $ is sufficiently small.

Moreover, there is a $\rho >0$ so that for all $x\in B\left( p_{i},\rho
\right) \cup B\left( A\left( p_{i}\right) ,\rho \right) ,$ the index $j$ in (%
\ref{der g is big}) can be chosen to be different from $i.$
\end{corollary}

Since directions $v$ that satisfy (\ref{alom horiz for S}) are almost
horizontal for $\mathrm{pr}:\overline{A^{n-1}(0,r-\varepsilon ,2r)}%
\rightarrow \partial \left( D^{n-1}(0,r-\varepsilon )\right) =S^{n-2}.$
Lemma \ref{close deriv lemma} and Corollary \ref{expand to U cor} give us
the following result.

\begin{corollary}
\label{cor 6}There is a $\lambda >0$ so that for all $x^{\alpha }\in
U^{\alpha }$ and all $v\in H_{x^{\alpha }}^{g_{d}^{\alpha }},$ there is a $j$
so that the $j^{th}$--component function of $g_{d}^{\alpha }$ satisfies%
\begin{equation*}
\left\vert D_{v}\left( \left( g_{d}^{\alpha }\right) _{j}\right) \right\vert
>\lambda \left\vert v\right\vert ,
\end{equation*}%
provided $U$ and $d$ are sufficiently small and $\alpha $ is sufficiently
large. In particular, $g_{d}^{\alpha }|_{U^{\alpha }}$ is a submersion.
\end{corollary}

Proposition \ref{g submersion} follows from Corollary \ref{cor 6}.

Let $p_{n}\in \mathcal{D}_{k}^{n}\left( r\right) $ be as in (\ref{dfn of p_i
dfn}). That is, {}%
\begin{equation*}
p_{n}\equiv \left\{ 
\begin{array}{ll}
\cosh (r)e_{0}+\sinh (r)e_{n} & \text{if }k=-1 \\ 
e_{0}+re_{n} & \text{if }k=0 \\ 
\cos (r)e_{0}-\sin (r)e_{n} & \text{if }k=1.%
\end{array}%
\right.
\end{equation*}

Let $Q:\mathcal{D}_{k}^{n}\left( r\right) \longrightarrow P_{k,r}^{n}$ be
the quotient map. We abuse notation and refer to $Q\left( p_{n}\right) $ as $%
p_{n}.$ We define $f_{n}:P_{k,r}^{n}\rightarrow \mathbb{R}$ by 
\begin{equation*}
f_{n}(x)\equiv h_{k}\circ \mathrm{dist}\left( p_{n},x\right) -h_{k}\circ 
\mathrm{dist}\left( p_{0},x\right) .
\end{equation*}%
With a slight modification of the proof of Proposition \ref%
{posdeltastrainedrad}, we get

\begin{lemma}
\label{strainers in level}There are $\delta ,r>0$ so that for all $x\in
E_{D}\left( r-\frac{\varepsilon }{2}\right) $, there is an $\left( n,\delta
,r\right) $--strainer $\left\{ \left( a_{i},b_{i}\right) \right\} _{i=1}^{n}$
for a neighborhood of $x$ with 
\begin{equation*}
\left\{ \left( a_{i},b_{i}\right) \right\} _{i=1}^{n-1}\subset
f_{n}^{-1}\left( l\right)
\end{equation*}%
for some $l\in \mathbb{R}.$
\end{lemma}

We cover $E_{D}\left( r-\frac{\varepsilon }{2}\right) $ by a finite number
of such $\left( n,\delta ,r\right) $--strained sets and make the following
definition.

\begin{definition}
For $x\in E_{D}\left( r-\frac{\varepsilon }{2}\right) ,$ set%
\begin{equation*}
H_{x}^{\Psi }\equiv \mathrm{span}_{i\in \left\{ 1,\ldots ,n-1\right\}
}\left\{ \uparrow _{x}^{a_{i}}\right\} ,
\end{equation*}%
where $\left\{ \left( a_{i},b_{i}\right) \right\} _{i=1}^{n-1}$ is as in the
previous lemma.
\end{definition}

Since $\Psi :E_{D}\left( r-\frac{\varepsilon }{2}\right) \longrightarrow
D^{n-1}\left( r-\frac{\varepsilon}{2}\right) $ is simply orthogonal
projection, we have

\begin{lemma}
There is a $\lambda >0$ so that for all $x\in E_{D}\left( r-\frac{%
\varepsilon }{2}\right) $ and all $v\in H_{x}^{\Psi },$ there is an $i$ so
that 
\begin{equation*}
\left\vert D_{v}f_{i}\right\vert >\lambda \left\vert v\right\vert .
\end{equation*}
\end{lemma}

To lift this lemma to the $M^{\alpha }$s, we need a notion of $\Psi
_{d}^{\alpha }$--almost horizontal for each $M^{\alpha }.$ Given $z^{\alpha
}\in E_{D}^{\alpha }\left( r-\frac{\varepsilon }{2}\right) ,$ we define a $%
\Psi _{d}^{\alpha }$--almost horizontal space at $z^{\alpha }$ as follows.
Let $B^{\alpha }$ be an $\left( n,\delta ,r\right) $--strained neighborhood
for $z^{\alpha }$ with strainers $\left\{ \left( a_{i}^{\alpha
},b_{i}^{\alpha }\right) \right\} _{i=1}^{n}$ that converge to $\left\{
\left( a_{i},b_{i}\right) \right\} _{i=1}^{n},$ where $\left( B,\left\{
\left( a_{i},b_{i}\right) \right\} _{i=1}^{n}\right) $ is part of our finite
collection of $\left( n,\delta ,r\right) $--strained neighborhoods for
points in $E_{D}\left( r-\frac{\varepsilon }{2}\right) $ that comes from
Lemma \ref{strainers in level}. We set 
\begin{equation*}
H_{z^{\alpha }}^{\Psi _{d}^{\alpha }}\equiv \mathrm{span}_{i\in \left\{
1,\ldots ,n-1\right\} }\left\{ \uparrow _{z^{\alpha }}^{a_{i}^{\alpha
}}\right\} ,
\end{equation*}%
where $\uparrow _{z^{\alpha }}^{a_{i}^{\alpha }}$ is the direction of \emph{%
some }segment from $z^{\alpha }$ back to $a_{i}^{\alpha }.$ Regardless of
these choices, $H_{z^{\alpha }}^{\Psi _{d}^{\alpha }}$ satisfies the
following lemma, whose proof is nearly identical to the proof of Corollary %
\ref{cor 6}.

\begin{lemma}
\label{Pis on disk bundle lemma}There is a $\lambda >0$ so that for all $%
z^{\alpha }\in E_{D}^{\alpha }\left( r-\frac{\varepsilon }{2}\right) $ and
all $v\in H_{z^{\alpha }}^{\Psi _{d}^{\alpha }},$ there is an $i\in \left\{
1,\ldots ,n-1\right\} $ so that 
\begin{equation*}
\left\vert D_{v}f_{i,d}^{\alpha }\right\vert >\lambda \left\vert
v\right\vert ,
\end{equation*}%
provided $\alpha $ is sufficiently large and $d$ is sufficiently small. In
particular, $\Psi _{d}^{\alpha }|_{E_{0}^{\alpha }\left( \varepsilon
/2\right) }$ is a submersion.
\end{lemma}

\begin{proposition}
\label{Param Stabil}$E_{A}^{\alpha }\left( r-\varepsilon \right) $ is
homeomorphic to $S^{n-2}\times D^{2},$ and $E_{D}^{\alpha }\left(
r-\varepsilon \right) $ is homeomorphic to $D^{n-1}\times S^{1},$ provided $%
\alpha $ is sufficiently large and $d$ is sufficiently small.
\end{proposition}

\begin{proof}
First we show that $E_{D}^{\alpha }\left( r-\varepsilon \right) $ is
connected. By the Stability Theorem \cite{Kap}, we have homeomorphisms $%
h_{\alpha }:P_{k}^{n}\left( r\right) \longrightarrow M^{\alpha }$ that are
also Gromov--Hausdorff approximations (cf. \cite{GrovPet1}, \cite{GrovPet3}
and \cite{Perel}). Thus for $\alpha $ sufficiently large, we have 
\begin{equation*}
E_{D}^{\alpha }\left( r-\varepsilon \right) \subset h_{\alpha }\left(
E_{D}\left( r-\frac{\varepsilon }{2}\right) \right) .
\end{equation*}%
Define $\rho ^{\alpha }:M^{\alpha }\longrightarrow \mathbb{R}$ by 
\begin{equation*}
\rho ^{\alpha }\left( x\right) \equiv \left\vert \Psi _{d}^{\alpha }\left(
x\right) \right\vert .
\end{equation*}%
Since $\Psi _{d}^{\alpha }|_{E_{D}^{\alpha }\left( r-\frac{\varepsilon }{2}%
\right) }$ is a submersion, it follows that $\rho ^{\alpha }$ does not have
critical points on $E_{D}^{\alpha }\left( r-\frac{\varepsilon }{2}\right)
\setminus E_{D}^{\alpha }\left( r-2\varepsilon \right) .$

There is a one-to-one correspondence between the flow lines of $\nabla \rho
^{\alpha }$ and the boundary of $E_{D}^{\alpha }\left( r-\varepsilon \right)
.$ Since each point of $h_{\alpha }\left( E_{D}\left( r-\frac{\varepsilon }{2%
}\right) \right) $ is on precisely one flow line, the flow lines of $\nabla
\rho ^{\alpha }$ give a continuous map from $h_{\alpha }\left( E_{D}\left( r-%
\frac{\varepsilon }{2}\right) \right) $ onto $E_{D}^{\alpha }\left(
r-\varepsilon \right) .$ In particular, $E_{D}^{\alpha }\left( r-\varepsilon
\right) $ is connected.

Since the domain of $\Psi _{d}^{\alpha }|_{E_{D}^{\alpha }\left(
r-\varepsilon \right) }$ is compact, $\Psi _{d}^{\alpha }|_{E_{D}^{\alpha
}\left( r-\varepsilon \right) }$ is proper. Since it is also a submersion,
it is a fiber bundle with contractible base $D^{n-1}\left( 0,r-\varepsilon
\right) .$ Since the fiber is $1$--dimensional and the total space is
connected, we conclude that $E_{D}^{\alpha }\left( r-\varepsilon \right) $
is homeomorphic to $D^{n-1}\times S^{1}.$ Since $E_{D}\left( r-\frac{%
\varepsilon }{2}\right) $ is also homeomorphic to $D^{n-1}\times S^{1},$
there is a homeomorphism $h_{0}:E_{D}\left( r-\frac{\varepsilon }{2}\right)
\longrightarrow E_{D}^{\alpha }\left( r-\frac{\varepsilon }{2}\right) $ so
that 
\begin{center}
\begin{tikzpicture}[description/.style={fill=white,inner sep=2pt}]
\matrix (m) [matrix of math nodes, row sep=3em,
column sep=2.5em, text height=1.5ex, text depth=0.25ex]
{ E_0(r-
\frac{\varepsilon}{2}) & & E_0^{\alpha}(r-\frac{\varepsilon}{2}) \\
& D^{n-1} & \\ };
\path[->,font=\scriptsize]
(m-1-1) edge node[auto] {$ h_0 $} (m-1-3)
edge node[auto,swap] {$ \Psi_d $} (m-2-2)
(m-1-3) edge node[auto] {$ \Psi_d^{\alpha} $} (m-2-2);
\end{tikzpicture}
\end{center}
commutes. Using the Strong Gluing Theorem (\cite{Kap}%
, Theorem 4.10), and the fact that $\Psi _{d}^{\alpha }$ converges to $\Psi $
as $\alpha \rightarrow \infty $ and $d\rightarrow 0,$ we choose the
homeomorphism $h_{0}:E_{D}\left( r-\frac{\varepsilon }{2}\right)
\longrightarrow E_{D}^{\alpha }\left( r-\frac{\varepsilon }{2}\right) $ so
that it is a $\tau \left( \frac{1}{\alpha }\right) $-Gromov-Hausdorff
approximation.

Applying the Gluing Theorem again, we construct a homeomorphism $%
h:P_{k}^{n}\left( r\right) \longrightarrow M^{\alpha }$ so that 
\begin{equation*}
h=\left\{ 
\begin{array}{cc}
h_{0} & \text{on }E_{D}\left( r-\varepsilon \right) \\ 
h_{\alpha } & \text{on }E_{A}\left( r-\frac{\varepsilon }{4}\right) .%
\end{array}%
\right.
\end{equation*}%
It follows that $h\left( E_{A}\left( r-\varepsilon \right) \right)
=E_{A}^{\alpha }\left( r-\varepsilon \right) .$ Since $E_{A}\left(
r-\varepsilon \right) $ is homeomorphic to $S^{n-2}\times D^{2},$ the result
follows.
\end{proof}

We are now in a position to prove the Disk and Circle Bundle Lemmas.

\begin{proof}[Proof of the Disk Bundle Lemma]
By Proposition \ref{g submersion}, $g_{d}^{\alpha }:E_{A}^{\alpha }\left(
r-\varepsilon \right) \longrightarrow $ $\partial D^{n-1}(0,r-\varepsilon
)=S^{n-2}$ is a submersion. Since the domain of $g_{d}^{\alpha }$ is
compact, $g_{d}^{\alpha }$ is proper. So $g_{d}^{\alpha }$ is a fiber bundle
with two-dimensional fiber $F.$ From the long exact homotopy sequence and
Proposition \ref{Param Stabil}, we conclude that $F$ is a $2$--disk. The
orientations of $E_{A}^{\alpha }\left( r-\varepsilon \right) \cong
S^{n-2}\times D^{2}$ and $S^{n-2}$ together induce an orientation on the
fibers of $E_{A}^{\alpha }\left( r-\varepsilon \right) .$ The oriented $2$%
--disk bundles over $S^{n-2}$ are classified by $\pi _{n-3}\left( \mathrm{%
Diff}_{+}\left( D^{2}\right) \right) ,$ where $\mathrm{Diff}_{+}\left(
D^{2}\right) $ is the group of orientation preserving diffeomorphisms of $%
D^{2}.$ By Theorem 1 of \cite{LB}, $\pi _{n-3}\left( \mathrm{Diff}_{+}\left(
D^{2}\right) \right) \cong \pi _{n-3}\left( SO\left( 2\right) \right) \cong
\left\{ 0\right\} $, unless $n=4.$ So for $n\neq 4,$ every $D^{2}$--bundle
over $S^{n-2}$ is trivial.

When $n=4$, $E_{A}\left( r-\varepsilon \right) $ is a $D^{2}$--bundle over $%
S^{2}$ whose total space is homeomorphic to $S^{2}\times D^{2}.$ The $D^{2}$%
--bundles over $S^{2}$ are precisely those whose corresponding unit circle
bundles are lens spaces. (See for example \cite{Steen}, page 135.) Since the
total space of $E_{A}\left( r-\varepsilon \right) $ is homeomorphic to $%
S^{2}\times D^{2},$ it follows that $E_{A}\left( r-\varepsilon \right) $ is
trivial in all cases, completing the proof of the Disk Bundle Lemma.
\end{proof}

\begin{proof}[Proof of the Circle Bundle Lemma]
Since $\Psi _{d}^{\alpha }|_{E_{0}^{\alpha }\left( \varepsilon \right) }$ is
a proper submersion, 
\begin{equation*}
\left( E_{D}^{\alpha }\left( r-\varepsilon \right) ,\Psi _{d}^{\alpha
}\right)
\end{equation*}
is a fiber bundle over $D^{n-1}\left( 0,r-\varepsilon \right) $ with
one-dimensional fiber $F.$ Since $E_{D}^{\alpha }\left( r-\varepsilon
\right) $ is also homeomorphic to $D^{n-1}\times S^{1},$ it follows that the
fiber is $S^{1}$. The base is contractible, so the bundle is trivial.
\end{proof}

This completes the proofs of Theorem \ref{Gromoll}, Theorem \ref{Purse Stab}%
, and the Main Theorem.

\end{document}